\documentclass[12pt]{article}

\usepackage{amsmath,amsfonts,amssymb,float}
\usepackage{alltt}
\usepackage{tablefootnote}

\usepackage{amsmath,amsfonts,ifthen,fullpage,enumerate}  

\usepackage{amsmath,amsfonts,amssymb}
\usepackage{hyperref}

\usepackage{makecell}  

\usepackage{graphicx}        

\usepackage[table]{xcolor}   

\def\spam{\mathop{\rm span}\nolimits}

\def\grey{\cellcolor{gray!25}}

\newcommand{\figures}{}

\def\trace{\mathop{\rm trace}\nolimits}

\def\diag{\mathop{\rm diag}\nolimits}

\def\pmat#1{\begin{pmatrix}#1\end{pmatrix}}

\def\question#1{{\bf Question: }#1}
\def\question#1{}

\def\R{\mathbb{R}}

\def\ZZ{\mathbb{Z}}

\def\Cd{\C^d}
\def\Rd{\R^d}

\def\C{\mathbb{C}}

\def\SS{\mathbb{S}}

\newcommand{\RR}{\mathbb{R}}

\newtheorem{theorem}{Theorem}[section]

\newtheorem{example}{Example}[section]

\newenvironment{proof}{{\noindent \it
Proof.}}{\hfill$\Box$\medskip}
\newif\ifdraft\def\draft{\drafttrue\hoffset=.8truecm\showlabeltrue
\def\comment##1{{\bf comment: ##1}}
\headline={\sevenrm \hfill \ifx\filenamed\undefined\jobname\else\filenamed\fi%
(.tex) (as of \ifx\updated\undefined???\else\updated\fi)
 \TeX'ed at {\hour\time\divide\hour by 60{}%
\minutes\hour\multiply\minutes by 60{}%
\advance\time by -\minutes
\the\hour:\ifnum\time<10{}0\fi\the\time\  on \today\hfill}}
}

\def\inpro#1{\langle#1\rangle}
\def\ip#1{\langle\kern-.28em\langle#1\rangle\kern-.28em\rangle_\nu}

\def\norm#1{\Vert#1\Vert}
\def\openR{{{\rm I}\kern-.16em {\rm R}}}

\let\ga\alpha

\let\gb\beta

\let\gD\Delta

\let\gth\theta

\let\gL\Lambda

\let\gs\sigma

\let\ga\alpha

\let\gb\beta

\let\gs\sigma
\def\inpro#1{\langle#1\rangle}

\def\Hom{\mathop{\rm Hom}\nolimits}

\def\Implies{\hskip1em\Longrightarrow\hskip1em}

\def\formeq{\the\sectionno.\the\equationno}  
\def\elabel#1/#2/#3/{\global\advance\equationno by 1 %
\ifx#1\empty\else\emember#1%
\ifshowlabel\marginal{\string#1}\fi\fi%
\ifmmode\eqno{#3(\formeq#2)}\else#3\formeq#2\fi} 

\def\makeblanksquare#1#2{
\dimen0=#1pt\advance\dimen0 by -#2pt
      \vrule height#1pt width#2pt depth0pt\kern-#2pt
      \vrule height#1pt width#1pt depth-\dimen0 \kern-#1pt
      \vrule height#2pt width#1pt depth0pt \kern-#2pt
      \vrule height#1pt width#2pt depth0pt
}

\def\damntilde{\char126}  

\title{\bf 
Putatively optimal projective spherical designs 
with little apparent symmetry
}

\author{Alex Elzenaar and Shayne Waldron\\
 \\
Department of Mathematics \\ University of Auckland\\
Private
Bag 92019, Auckland, New Zealand\\
e-mail: waldron@math.auckland.ac.nz}
\date{\today}

\begin{document}

\maketitle 

\begin{abstract}
We give some new explicit examples of putatively optimal projective 
spherical designs. 
i.e., ones for which there is numerical evidence that they are of minimal size.
These form continuous families, and so have little apparent symmetry 
in general, which 
requires the introduction of new techniques for their construction.
New examples of interest include an $11$-point spherical $(3,3)$-design for
$\RR^3$, and a $12$-point spherical $(2,2)$-design for $\RR^4$
given by four Mercedes-Benz frames that lie on equi-isoclinic planes.
We also give results of an extensive numerical study to determine the 
nature of the real algebraic variety of optimal projective real spherical designs,
and in particular when it is a single point (a unique design) 
or corresponds to an infinite family of designs.
\end{abstract}

\bigskip
\vfill

{\tt Manopt} software,
real algebraic variety

 \ifdraft (Other designs, configurations) \else\fi
65D30, \ifdraft (Numerical integration) \else\fi
65K10, \ifdraft (Numerical optimization and variational techniques) \else\fi
49Q12, \ifdraft (Sensitivity analysis for optimization problems on manifolds) \else\fi
65H14, \ifdraft Numerical algebraic geometry \else\fi

\quad
secondary
14Q10, \ifdraft	Computational aspects of algebraic surfaces \else\fi
14Q65, \ifdraft	Geometric aspects of numerical algebraic geometry \else\fi
42C15, \ifdraft General harmonic expansions, frames  \else\fi
94B25.  \ifdraft Combinatorial codes  \else\fi

\vskip .5 truecm
\hrule
\newpage

\section{Introduction}

Due to a wide range of applications,
there is a large body of work
on the general problem of constructing points (or lines) on a sphere 
which are optimally separated in some way.
These configurations can be numerical or explicit, 
with the general hope being that numerical configurations
of interest approximate explicit constructions that might be found.
Some examples include Hardin and Sloane's list of numerical
spherical $t$-designs \cite{HS96}, the numerical constructions
of Weyl-Heisenberg SICs ($d^2$ equiangular lines in $\Cd$)
\cite{SG09} and exact constructions obtained from them
\cite{ACFW18}, the ``Game of Sloanes'' optimal packings
in complex projective space \cite{JKM19}, and minimisers
of the $p$-frame energy on the sphere \cite{BGMPV19}.

Here we consider numerical and explicit constructions 
of a putatively optimal 
set of points (or lines)
of what are variously called spherical $(t,t)$-designs for
$\Rd$ \cite{HW21}, spherical half-designs \cite{KP11} 
and projective $t$-designs \cite{H82}.
These are given by a sequence of vectors $v_1,\ldots,v_n\in\Rd$
(not all zero) which give equality in the inequality
\begin{equation} \label{realandcomplexWelchbd}
\sum_{j=1}^n \sum_{k=1}^n |\inpro{v_j,v_k}|^{2t} \ge 
 {1\cdot3\cdot5\cdots (2t-1)\over d(d+2)\cdots(d+2(t-1))}
\Bigl(\sum_{\ell=1}^n \norm{v_\ell}^{2t}\Bigr)^2,
\end{equation}
where $t=1,2,\ldots$. 
The case where all the vectors have unit length is variously referred
to as an equal-norm/unweighted/classical design, and in general
as a weighted design.
We observe (see \cite{W17}, \cite{HW21}) that
\begin{itemize}
\item These are projective objects (lines), which are 
counted up to projective unitary equivalence, i.e., for $U$ unitary
and $c_j$ unit scalars, we have that $(v_j)$ is a spherical $(t,t)$-design
if and only if $(c_jUv_j)$ is, and these are considered to be equivalent.
\item Spherical $(t,t)$-designs of $n$ vectors in $\Rd$ exist 
for $n$ sufficiently large, i.e., the algebraic variety
given by (\ref{realandcomplexWelchbd}) is nonempty for $n$ sufficiently large.
Designs for which $n$ is minimal are of interest, and are said to be {\bf optimal}.
\item The existence of (optimal) spherical designs can investigated 
numerically. 
\end{itemize}

If $(v_j)$ gives equality in (\ref{realandcomplexWelchbd}) up to 
machine precision, then we will call it a {\bf numerical} 
design. We say a numerical or explicit design is {\bf putatively
optimal} if a numerical search (which finds it) suggests that there 
is no design with fewer points.

The examples of putatively optimal spherical $(t,t)$-designs for $\Rd$ found 
so far (see Table 6.1 of \cite{W18})
come from cases where the algebraic variety of spherical $(t,t)$-designs
(up to equivalence) appears to consist of a finite number of points.
This can be detected by considering the $m$-products
$$ \gD(v_{j_1},\ldots,v_{j_m})
:= \inpro{v_{j_1},v_{j_2}} \inpro{v_{j_2},v_{j_3}}\cdots \inpro{v_{j_m},v_{j_1}},
\qquad 1\le j_1,\ldots,j_m\le n, $$
which determine projective unitary equivalence \cite{CW15}.
From these, it is then possible to conjecture what the symmetry
group of the 
design is \cite{CW18}, and ultimately to 
construct an explicit (putatively optimal) spherical $(t,t)$-design
as the orbit of a few vectors under the unitary action of the symmetry group
(cf.\ \cite{HW20}, \cite{ACFW18}).

In this paper, we consider, for the first time, the case when the
algebraic variety of optimal spherical $(t,t)$-designs appears 
to be uncountable (of positive dimension).
In the examples that we consider, a generic numerical putatively
optimal spherical $(t,t)$-design has a trivial symmetry group. 
However, there is often some structure, 
referred to as ``{\it repeated angles}'', i.e., some $2$-products 
$$ \gD(v_j,v_k)
= |\inpro{v_j,v_k}|^2, \qquad j\ne k, $$
are repeated. This is just enough structure to tease
out an uncountably infinite family of putatively optimal 
spherical $(t,t)$-designs, in some examples.

\section{Numerics}

For $V=[v_1,\ldots,v_n]\in\RR^{d\times n}$, let
$f(V)=f_{t,d,n}(V)\ge0$ be given by
\begin{equation} 
\label{fdefn}
f(V)
:=\sum_{j=1}^n \sum_{k=1}^n |\inpro{v_j,v_k}|^{2t} - c_t(\Rd)
\Bigl(\sum_{\ell=1}^n \norm{v_\ell}^{2t}\Bigr)^2, \qquad
c_t(\Rd):= \prod_{j=0}^{t-1} {2j+1\over d+2j}.
\end{equation}

We consider the real algebraic variety of spherical $(t,t)$-designs 
given by $f(V)=0$, subject to the (algebraic) constraints
\begin{align*} 
& \norm{v_1}^2=\cdots=\norm{v_n}^2=1, 
\qquad\hbox{equal-norm/unweighted/classical designs} \cr
&\norm{v_1}^2+\cdots+\norm{v_n}^2=n,
\qquad\hbox{weighted designs \quad
($n$ chosen for convenience).}
\end{align*}
This has been studied in the case $t=1$, where it 
gives the {\it tight frames} \cite{CMS17}, \cite{W18}.
In particular, local minimisers of $f$ for $t=1$ are global minimisers.
It is not known if this is true for $t>1$, and obviously this impacts
on the numerical search for designs, e.g., a local minimiser which 
was not a global minimiser might be more easily found, leading to 
a false conclusion that there is no spherical $(t,t)$-design.

We are primarily interested in the minimal $n$ for which 
the variety is nonempty (denoted by $n_e$ and $n_w$, respectively), i.e.,
the optimal spherical $(t,t)$-designs. We have
$$ {t+d-1\choose t}=\dim(\Hom(t))\le n_w\le n_e \le \dim(\Hom(2t))
={2t+d-1\choose 2t}. $$
For $d$ fixed, $n_e$ and $n_w$ are increasing functions of $t$.

A numerical search was done in \cite{HW21} using an iterative method 
that moves in the direction of $-\nabla f(V)$. 
The results there, and in Table 1 of \cite{BGMPV19}, have
been duplicated and extended by using the {\tt manopt} software \cite{manopt}
for optimisation on manifolds and matrices
(implemented in Matlab). The putatively optimal numerical designs that we found
are summarised in Table \ref{real(t,t)-designs},
and can be downloaded from \cite{EW22} and viewed at
$$ \hbox{\tt www.math.auckland.ac.nz/\damntilde waldron/SphericalDesigns} $$

Here are some details about our {\tt manopt} calculations:
\begin{itemize}
\item The cost function $f$ of (\ref{fdefn}) was minimised using the 
{\tt trustregions} solver.
\item This requires the manifold over which the minimisation is done 
to be specified.
We used {\tt obliquefactory} for real equal-norm designs and
{\tt euclideanfactory} for real weighted designs,
and {\tt obliquecomplexfactory} and {\tt euclideancomplexfactory}
for complex designs.
\item Since {\tt euclideanfactory(d,n)} is the manifold $\RR^{d\times n}$,
minimising the homogeneous polynomial $f$ tended to give minima of 
small norm. To avoid this, we added the term $(\norm{v_1}^2-1)^2$ to the 
cost function, so that the weighted designs $V=[v_1,\ldots,v_n]$ obtained have the 
first vector $v_1$ of unit norm.
For the purpose of calculating errors, $V$ was normalised so that
$\norm{v_1}^2+\cdots+\norm{v_n}^2=n$
(as for unit-norm designs).
\item
The solver requires the gradient and Hessian of $f$ as parameters.
The gradient function (page 140, \cite{W18}) was given 
explicitly, and the Hessian was calculated symbolically from 
$f$ by {\tt trustregion}.
\item We used the default solver options, except for the
{\tt delta\textunderscore bar} parameter, where
setting {\tt problem.delta\textunderscore bar}
to {\tt problem.M.typicaldist()/10}, rather than the default
{\tt problem.M.typicaldist()} gave better results.
\item
We considered the absolute error in $V$ being a design, i.e.,
\begin{equation}
f_{t,d,n} = f_{t,d,n}(V)
 :=\hbox{$\sum_j\sum_k |\inpro{v_j,v_k}|^{2t}- c_t(\Rd) 
(\sum_\ell\norm{v_\ell}^{2t})^2\ge0$}, 
\end{equation}
where $\trace(V^*V)=\norm{v_1}^2+\cdots+\norm{v_n}^2=n$.
\end{itemize}
See \cite{E20} for further details. 


\begin{figure}[ht]
\centering
\includegraphics[width=7.5cm]{\figures 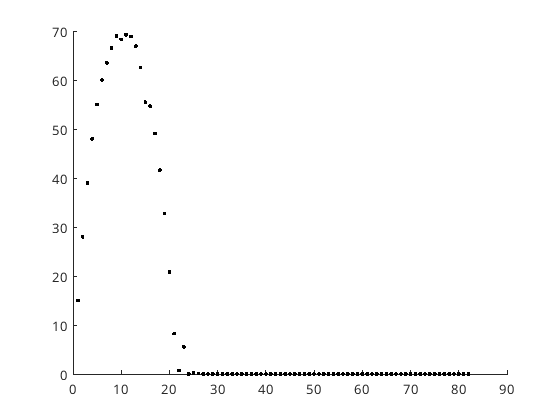}
\includegraphics[width=7.5cm]{\figures 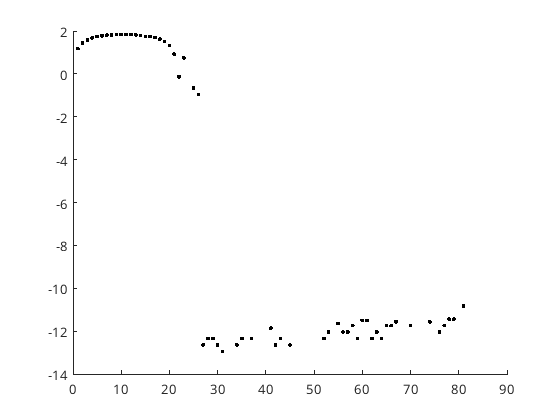}
%
%
\caption{The graphs of $n\mapsto f_{t,d,n}$ and $n\mapsto \log_{10} f_{t,d,,n}$
for $t=2$, $d=6$, i.e., the error in numerical approximations to a  
unit-norm spherical $(2,2)$-design of $n$ vectors in $\RR^6$.  }
\label{generic-situation}       
\end{figure}

We now discuss the heuristics of determining 
when $f_{t,d,n}(V)$ is (numerically) zero.

\section{The overall picture}

We use $f_{t,d,n}(V)$ for a numerically computed $V=[v_1,\ldots,v_n]$ 
as a proxy for
$$ \ga_{t,d,n} := \min_{V\in\RR^{d\times n}\atop\trace(V^*V)=n} f_{t,d,n}(V),$$
where the condition $\norm{v_j}=1$ is added for unit-norm designs.
It is known that
\begin{itemize} 
\item For equal-norm designs
$n\mapsto \ga_{t,d,n}$ is zero for some (large) $n$.
\item For unweighted designs
$n\mapsto \ga_{t,d,n}$ is decreasing, becoming zero for some (large) $n$.
\end{itemize}
Moreover
\begin{itemize}
\item For large $n$ (relative to $t$ and $d$), a random set of $n$ points 
is close to being a spherical $(t,t)$-design, 
and hence has a small error $f_{t,d,n}(V)$. 
\end{itemize}
A priori, these properties suggest that it may be difficult to identify 
$(t,t)$-designs, in the sense that the error $n\mapsto f_{t,d,n}(V)$ 
slowly approaches numerical zero.
However, extensive calculations suggest that in
the ``generic'' situation (see Figure \ref{generic-situation})
this is not the case:
\begin{itemize}
\item[] {\bf Generic situation:} At the point where an optimal $(t,t)$-design
is obtained the error ``jumps down'' to numerical zero.
\end{itemize}


\noindent
There are also ``special'' situations 
(see Figures \ref{special-situation} and \ref{special-situation_3_8}), 
where 
(by reasons of symmetry)
\begin{itemize}
\item[] {\bf Special situation:}
An equal-norm 
$(t,t)$-design with a unexpectedly small number of vectors
exists. This design may or may not be obtained by 
calculating a single numerical design. Here the error jumps to zero, 
but then returns to roughly the generic situation (nonzero with an eventual
jump to numerical zero).
\end{itemize}

\begin{figure}[ht]
\centering
\includegraphics[width=7.5cm]{\figures 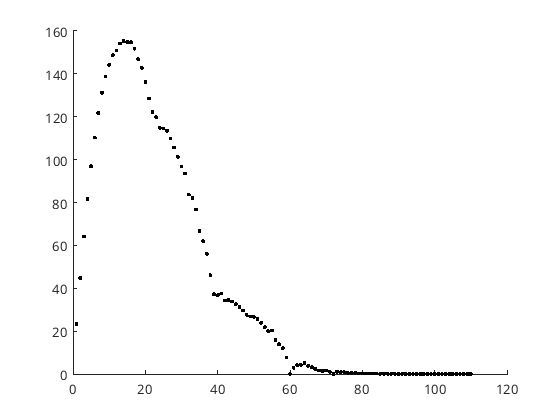}
\includegraphics[width=7.5cm]{\figures 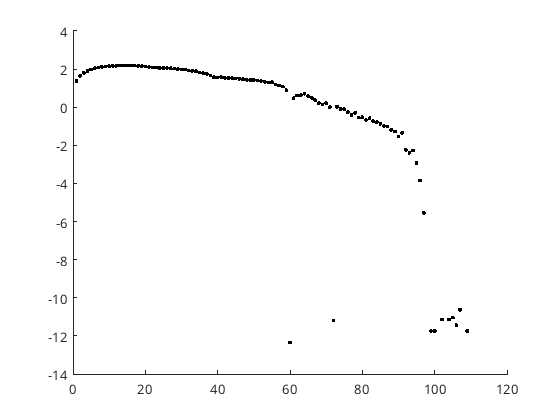}
\caption{The graphs of $n\mapsto f_{t,d,n}$ and $n\mapsto \log_{10} f_{t,d,,n}$
for $t=5$, $d=4$, i.e., the error in numerical approximations to a
unit-norm spherical $(5,5)$-design of $n$ vectors in $\RR^4$.  }
\label{special-situation}       
\end{figure}

The error graphs for unweighted $(t,t)$-designs share this ``jump'' phenomenon
(see Figure \ref{weighted-jump}),
but are strictly decreasing (becoming constant once zero is obtained).
This is because a zero weight corresponds to a design
with one fewer point (and so increasing the number of points enlarges the 
possible set of designs).

\begin{figure}[H]
\centering
\includegraphics[width=7.5cm]{\figures 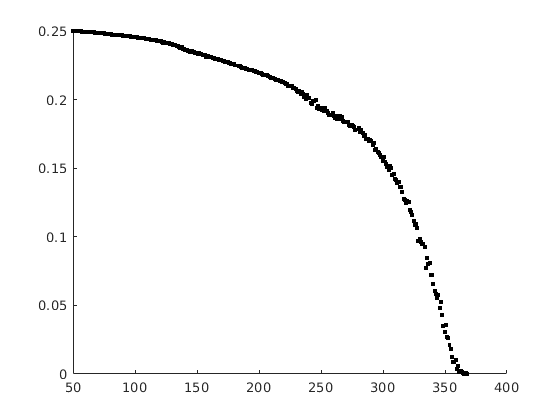}
\includegraphics[width=7.5cm]{\figures 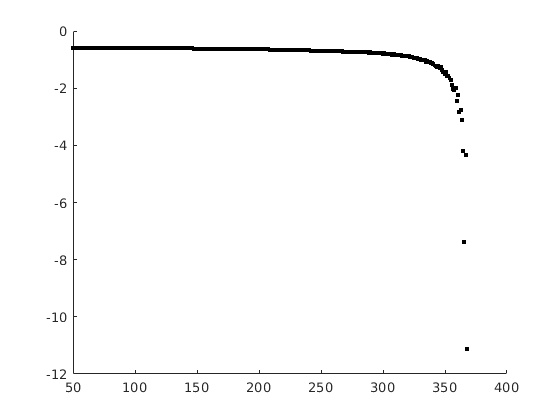}
\caption{The graphs $n\mapsto f_{t,d,n}$ and $n\mapsto \log_{10} f_{t,d,,n}$
of the error in 
approximations to weighted designs with 
$t=6$ and $d=5$, i.e., 
$(6,6)$-designs of
$n$ vectors in $\RR^5$.}
\label{weighted-jump}       
\end{figure}

The cost of finding of a numerical approximation to a spherical 
$(t,t)$-design in $\Rd$ grows with $t$ and $d$.
Therefore (like in previous studies)
we could only calculate numerical designs up to a certain point. 
The previous calculations of \cite{BGMPV19} and  \cite{HW21}
were replicated and extended. These are summarised in Table 
\ref{real(t,t)-designs} below, with 
comments, e.g.,
\begin{itemize}
\item[]  {\it structure} means some angles are repeated for equal-norm designs
({\it repeated angles}),
and some norms are repeated for unweighted designs.
\item[] {\it infinite family} means a different numerical design is obtained
each time, and we infer that the algebraic variety of optimal designs
has positive dimension.
\item[] {\it group structure} means that a finite number of numerical 
designs are obtained, which are a union of orbits of some (symmetry) group.
\end{itemize}
A set of equal-norm vectors for which the angles $|\inpro{v_j,v_k}|$, $j\ne k$, 
are all equal is said to be {\bf equiangular}.

The following example shows that
minimising $f_{t,d,n}$ over a larger number of points than for an optimal 
design can give a unique configuration.


\begin{example}
\label{Mercedesexample} Minimisation of $f_{t,d,n}$ 
for $t=2$ and $n$ equal-norm vectors in $\RR^2$ gives
\begin{itemize}
\item[] {\rm $n=3$:} 
the unique optimal configuration of three equiangular lines in $\RR^2$.
\item[] {\rm $n=4$:}
a unique configuration of two MUBs (mutually unbiased bases), 
\hbox{\hskip1.7truecm} 
 \hbox{\hskip1.2truecm}
equivalently, four equally spaced lines.
\item[] {\rm $n=5$:}
a unique configuration of five equally spaced lines.
\item[] {\rm $n=6$:}
configurations with six angles of ${1\over2}$ 
and three other angles (each appearing \hbox{\hskip1.3truecm} $3$ times), which are seen to be the
union of two Mercedes-Benz frames.
\end{itemize}
The set of $t+1$ equally spaced lines in $\RR^2$ is a known optimal
spherical $(t,t)$-design.
\end{example}

\vfil\eject

\setlength{\tabcolsep}{3pt}
\renewcommand{\arraystretch}{1.0}
\renewcommand{\arraystretch}{0.85}
\renewcommand{\arraystretch}{0.9}
\begin{table}[H]
\small
\footnotesize
\caption{The minimum numbers $n_w$ and $n_e$ of vectors in
a weighted and in a equal-norm spherical $(t,t)$-design for $\Rd$
(spherical half-design of order $2t$) as calculated numerically.
The $(t,t)$-design of $t+1$ vectors in $\RR^2$ was 
obtained for all $t$ (not all cases are listed).
 }
\begin{center}
\label{real(t,t)-designs}       
\begin{tabular}{| >{$}l<{$}  >{$}l<{$} | >{$}l<{$}  >{$}l<{$} | >{$}l<{$} >{$}l<{$} | }
\hline  
t & d & n_w & n_e & \hbox{Remarks on $n_w$} & \hbox{Remarks on $n_e$} \\
\hline
 & & & & & \\ [-2.0ex]
2 & 2 & 3 & 3 & \hbox{Mercedes-Benz frame} & \hbox{see Example \ref{Mercedesexample}} \\
2 & 3  & 6  & 6  &   \hbox{equiangular lines in $\RR^3$} & \\
2 & 4 & 11 & 12 & \hbox{\S\ref{Rez(2,2)-designs} \cite{S71}, \cite{R92}, infinite family}    
& \hbox{infinite family (Theorem \ref{isoclinicMercedes})} 
 \\
2 & 5 & 16 & 20 &  
\hbox{\S\ref{Rez(2,2)-designs} unique, group structure \cite{MW20}}  
& \hbox{infinite family (Example \ref{20point2desR5ex}) } \\
2 & 6 & 22  & 24  & \hbox{\S\ref{Rez(2,2)-designs} unique, group structure
\cite{MW20} } & \hbox{repeated angles (Example \ref{24point2desR6ex})} \\ 
2 & 7  & 28 & 28 &  \hbox{equiangular lines in $\RR^7$} & \\
2 & 8 & 45  & 51 & \hbox{infinite family, no structure} & 
\hbox{infinite family, no structure}  \\
2 & 9 & 55 & 67 & \hbox{infinite family, no structure} 
& \hbox{infinite family, no structure}  \\
2 & 10 & 76 & 85 & \hbox{infinite family, no structure} 
& \hbox{infinite family, no structure}  \\
2 & 11 & 96 & 106 & \hbox{infinite family, no structure} & \hbox{infinite family, no structure} \\
2 & 12 & 120 & 131 & \hbox{infinite family, no structure} & \hbox{infinite family, no structure} \\
2 & 13 & 146 & 159 & \hbox{infinite family, no structure} & \hbox{infinite family, no structure} \\
2 & 14 & 177 & 190 & \hbox{infinite family, no structure}  & \hbox{infinite family, no structure} \\
2 & 15 & 212 & 226 & & \hbox{infinite family, no structure} \\
2 & 16 & 250 & 267 & & \hbox{infinite family, no structure} \\
2 & 17 & 294 & 312 & & \hbox{infinite family, no structure} \\
2 & 18 & 342 & 362 & & \\
 & & & & & \\ [-2.0ex]
3 & 2  & 4  & 4  &  \hbox{two real mutually unbiased bases} & \hbox{see Example \ref{Mercedesexample}} \\
3 & 3 & 11  & 16 &  \hbox{\S\ref{Rez3.3.11} Reznick, no structure} & \hbox{infinite family, no structure} \\
3 & 4 & 23  & 24 &   \hbox{group structure (Example \ref{23vect(3,3)designR4})} & \hbox{infinite family (Example \ref{Kempnerexample})} \\
3 & 5 & 41 & 55 &  \hbox{group structure (Example \ref{41vect(3,3)designR5})} & 
\hbox{infinite family, no structure}  \\
3 & 6 & 63 & 96 & \hbox{unique, two orbits (Example \ref{union2designs})} &
\hbox{infinite family, no structure} \\
3 & 7 & 91 & 158 & \hbox{unique, two orbits (Example \ref{union2designs})} &
\hbox{infinite family, no structure}  \\
3 & 8 & 120 & 120 & \hbox{unique (Example \ref{Kotelinaexample}) } &
\hbox{see Figure \ref{special-situation_3_8}}  \\
3 & 9 & 338 & 380 & \hbox{infinite family, no structure}  & \hbox{infinite family, no structure} \\
 & & & & & \\ [-2.0ex]
4 & 2 & 5 & 5 & \hbox{Equally spaced lines} & \hbox{see Example \ref{Mercedesexample}}  \\
4 & 3 & 16 & 24 & \hbox{unique, two orbits (Example \ref{union2designs})} 
& \hbox{repeated angles (Example \ref{24vect(4,4)designR3})} \\
4 & 4 & 43 & 57 & \hbox{infinite family, no structure}  &
\hbox{infinite family, no structure} \\
4 & 5 & 101 & 126 & \hbox{infinite family, no structure} & \hbox{infinite family, no structure} \\
4 & 6 & 217 & 261 & & \\
4 & 7 & 433 & \grey 504 & & \\
 & & & & & \\ [-2.0ex]
5 & 2 & 6 & 6 & \hbox{Equally spaced lines}  & \hbox{see Example \ref{Mercedesexample}} \\
5 & 3 & 24 & 35 & \hbox{infinite family, no structure} & 
\hbox{infinite family, no structure} \\
5 & 4 & 60 & 60 & \hbox{unique, one orbit \cite{HW21}} & 
\hbox{see Figure \ref{special-situation} and Example \ref{isoclinicsubspacesdesign}} \\
5 & 5 & 203 & \grey 253 & & \\
5 & 6 & 503 & \grey 604 & \hbox{infinite family, no structure} & \\
 & & & & & \\ [-2.0ex]
6 & 3 & 32 & 47 & \hbox{infinite family, no structure} & 
\hbox{infinite family, no structure} \\
6 & 4 & 116 & 154 & \hbox{infinite family, no structure} &
\hbox{infinite family, no structure} \\
6 & 5 & 368 & 458 & & \\
 & & & & & \\ [-2.0ex]
7 & 3 & 41 & 61 & \hbox{unique (Example \ref{41vect(7,7)designR3})} & \hbox{infinite family, no structure} \\
7 & 4 & 173 & 229 &  \hbox{infinite family, no structure} & \\
 & & & & & \\ [-2.0ex]
8 & 3 & 54 & 78 & \hbox{infinite family, some structure} & \hbox{infinite family, no structure} \\
8 & 4 & \grey 249 & & & \\
 & & & & & \\ [-2.0ex]
9 & 3 & 70 & 97 & \hbox{infinite family, no structure} & \\
9 & 4 & 360 & & \hbox{unique, two orbits (Example \ref{union2designs})} & \\
 & & & & & \\ [-2.0ex]
10 & 3 & 89 & \grey 119 & \hbox{infinite family, no structure} & 
\hbox{see Example \ref{largetjumps}} \\
\hline
\end{tabular}
\end{center}
\end{table}

We now describe some specific $(t,t)$-designs that we obtained
during 
our calculations.

\section{A family of $12$-point spherical $(2,2)$-designs for $\RR^4$}

Putatively optimal unit-norm $12$-point spherical $(2,2)$-designs for $\RR^4$ are
easily found. These numerical designs appear to have trivial projective
symmetry group. However, they all have the feature:
\begin{itemize}
\item {\em Each vector/line makes an angle of ${1\over2}$ with two others,} 
\end{itemize}
i.e., each row and column of the Gramian has two entries of modulus ${1\over2}$ 
(up to machine precision). We now outline how we went from this
observation, to an infinite family of explicit putatively optimal designs
(Theorem \ref{isoclinicMercedes}). 

\begin{itemize}
\item The vector and the two making an angle ${1\over2}$ with it
were seen (numerically) to give three equiangular lines. 
\item These four sets of three equiangular lines, were seen to be
Mercedes-Benz frames, i.e., each lies in a $2$-dimensional subspace.
\item The four associated $2$-dimensional subspaces are equi-isoclinic
planes in $\RR^4$.
\end{itemize}

Let $V_1,\ldots,V_4\in\RR^{4\times2}$
have orthonormal columns. Then $P_j:=V_jV_j^*$
is the orthogonal projection onto the $2$-dimensional 
subspace of $\RR^4$ spanned by the columns of $V_j$.
These four subspaces (planes) are said to be {\bf equi-isoclinic} if
\begin{equation}
\label{isoclinicdefn}
P_jP_kP_j = \gs^2 P_j, \qquad j\ne k,
\quad\hbox{for some $\gs$}.
\end{equation}
There is a unique such configuration 
\cite{LS73}, \cite{ET06} (up to a unitary map) given by
\begin{equation}
\label{Vjdefn}
[V_1,V_2,V_3,V_4]= {1\over\sqrt{6}}
\pmat{ \sqrt{6}&0&\sqrt{2}&0&\sqrt{2}&0&\sqrt{2}&0\cr 
0&\sqrt{6}&0&\sqrt{2}&0&\sqrt{2}&0&\sqrt{2}\cr
0&0&-2&0&1&-\sqrt{3}&1&\sqrt{3} \cr
0&0&0&-2&\sqrt{3}&1&-\sqrt{3}&1}.
\end{equation}
A {\bf Mercedes-Benz frame} is a set of three equiangular vectors/lines 
in a $2$-dimensional subspace.


\begin{theorem}
\label{isoclinicMercedes}
Let $(v_j)$ consist of four Mercedes-Benz frames 
that lie in four equi-isoclinic planes in $\RR^4$. Then $(v_j)$ is a $12$-vector
spherical $(2,2)$-design for $\RR^4$.
\end{theorem}

\begin{proof} 
Let $M_j\in\RR^{2\times3}$ give a Mercedes-Benz frame (in $\RR^2)$, i.e., have the form
$$ M_j= [u_j,Ru_j,R^2u_j], \qquad  R=\pmat{\cos{2\pi\over3} &-\sin {2\pi\over3}
\cr \sin {2\pi\over3} & \cos{2\pi\over3}}
=\pmat{-{1\over2} &-{\sqrt{3}\over2}
\cr {\sqrt{3}\over2} & -{1\over2} }, \quad u_j=\pmat{\cos\gth_j\cr\sin\gth_j},  $$
and $V_j\in\RR^{4\times2}$ be given by (\ref{Vjdefn}).
Then all such $(v_j)$ are given up to projective unitary equivalence by  
$V=[V_1M_1,\ldots,V_4M_4]$.
The variational condition to be such a design is
\begin{equation} 
\label{12ptvarcdn}
\sum_{j=1}^{12} \sum_{k=1}^{12} |\inpro{v_j,v_k}|^{4} 
= {1\cdot3\over 4\cdot6}
\Bigl(\sum_{\ell=1}^{12} \norm{v_\ell}^{4}\Bigr)^2
={1\over 8}12^2=18,
\end{equation}
which we now verify by considering the $16$ blocks of the Gramian $V^*V=[(V_jM_j)^*V_kM_k]$.

The four diagonal blocks $(V_jM_j)^*V_jM_j=M_j^*(V_j^*V_j)M_j=M_j^*M_j$ are
the Gramian of a Mercedes-Benz frame, and so each contribute
$ 3\cdot1+6\cdot({1\over2})^4 ={27\over 8}$ to the 
left-hand side 
of the sum (\ref{12ptvarcdn}).
The off-diagonal blocks are all circulant (by a direct calculation)
$$  (V_jM_j)^*V_kM_k = \pmat{ a&b&c\cr c&a&b\cr b&c&a}, \qquad a^4+b^4+c^4={1\over8}. $$
Thus (\ref{12ptvarcdn}) holds as $4\cdot{27\over 8}+12\cdot{3\over8}=18$.
\end{proof}



Here are some further observations on this example:

\begin{itemize}
\item Our calculations suggest this gives the entire variety of optimal designs.
\item A simple calculation shows that $|\inpro{v_j,v_k}|$ 
can take any value in the interval $[0,{1\over\sqrt{3}}]$.
\item The optimal designs 
$V=V_{\gth_1,\gth_2,\gth_3,\gth_4}$ 
described in the proof
are a continuous
family (depending on three real parameters).  
It is believed that these are all such designs.
\item A generic design has no projective symmetries.
\item There are designs with projective symmetries. In particular,
$V_{0,{\pi\over2},{\pi\over2},{\pi\over2}}$ consists of three real MUBs 
(mutually unbiased orthonormal bases) for $\RR^4$, i.e., orthonormal bases for which vectors 
from different bases make an angle $|\inpro{v_j,v_k}|={1\over2}$, and has a projective symmetry group of order $576$.
These have the nice presentation
$$ [B_1,B_2,B_3] = {1\over\sqrt{2}}\pmat{ 
1&1&0&0& 1&1&0&0& 1&1&0&0 \cr
1&-1&0&0& 0&0&1&1& 0&0&1&1 \cr
 0&0&1&1& 1&-1&0&0& 0&0&1&-1 \cr
0&0&1&-1& 0&0&1&-1& 1&-1&0&0} $$
where (\ref{12ptvarcdn}) holds as $12\cdot1+(12\cdot8)\cdot({1\over2})^4+(12\cdot3)\cdot0^4=18$.
This design can also be constructed 
as a union of one or two orbits of the Shephard-Todd group $G(2,1,4)$
(see \cite{MW20}), the generating vectors being
$(1,1,0,0)$ and $(1,0,0,0)$, ${1\over2}(1,1,1,1)$.
\item The idea of decomposing a tight frame (design) as a union of smaller dimensional ones, 
as we have done here, is an old idea to understand and construct them. Here we have considered 
the subsets of vectors which form a regular simplex in $\RR^2$, which \cite{FJKM18}
call the {\it binder}.  
\end{itemize}


\vfil\eject

\section{Selected calculations}

\subsection{A family of $24$-point spherical $(4,4)$-designs for $\RR^3$}

A set of three equiangular vectors $(v_j)$ is said to be {\bf isogonal} if they
span a $3$-dimensional subspace, i.e., by appropriately multiplying the
vectors by $\pm1$ their Gramian has the form
$$\pmat{1&a&a\cr a&1&a\cr a&a&1}, \quad -{1\over2}< a < 1. $$
The limiting case $a=-{1\over2}$ gives a Mercedes-Benz frame 
and $a=1$ gives three equal lines. 
These can be viewed as a {\it lift} of a Mercedes-Benz frame to three dimensions
\cite{W18}.

Putatively optimal $24$-point spherical $(4,4)$-designs for $\RR^3$ are readily calculated,
and all appear to have the following structure:

\begin{itemize} {\em
\item Each is a union of $8$ sets of three isogonal lines.
\item Each set of isogonal lines is the lift of a Mercedes-Benz frame in
a fixed $2$-dimensional subspace.
\item This suggests an order three rotational symmetry.
} \end{itemize} 
We speculate that (up to projective unitary equivalence)
every design has the form:
$$V=[v_1,gv_1,g^2v_1,\ldots,v_8,gv_8,g^2v_8], 
$$ 
where
$$ g=\pmat{1&\cr&R}, \quad 
R=\pmat{-{1\over2} &-{\sqrt{3}\over2}\cr {\sqrt{3}\over2} & -{1\over2} },
\qquad v_j=\pmat{b_j\cr c_j}, \quad 
b_j\in\RR,\ c_j=\pmat{y_j\cr z_j}\in\RR^2. $$
The blocks of the Gramian have the (numerically observed) circulant form
$$ 
[v_k,gv_k,g^2v_k]^*
[v_j,gv_j,g^2v_j]
= \pmat{
\inpro{v_j,v_k}&\inpro{gv_j,v_k}&\inpro{g^2v_j,v_k}\cr
\inpro{g^2v_j,v_k}&\inpro{v_j,v_k}&\inpro{gv_j,v_k}\cr
\inpro{gv_j,v_k}&\inpro{g^2v_j,v_k}&\inpro{v_j,v_k}}. $$
In particular, since $|b_j|^2+\norm{c_j}^2=1$, the diagonal blocks are given by
$$\pmat{ 1&a_j&a_j \cr a_j&1&a_j \cr a_j&a_j&1},
\qquad a_j:=\inpro{v_j,gv_j}=b_j^2+(1-b_j^2)\inpro{{c_j\over\norm{c_j}},
R{c_j\over\norm{c_j}}}
={3\over2}(b_j^2-{1\over3}). $$

The definition $f(V)=0$ for being a design gives a polynomial of 
degree $16$ in the $24$ variables $b_j,c_j$. The condition $|b_j|^2+\norm{c_j}^2=1$
allows this to be effectively reduced to $16$ variables. We now indicate
how the characterisation of a design as a cubature rule allows us to 
obtain a system of lower degree polynomials.

A unit-norm sequence of $n$ vectors $(v_j)$ in $\Rd$ is a spherical $(t,t)$-design
if and only if it satisfies the cubature rule  (see Theorem 6.7 \cite{W18})
\begin{equation}
\label{cubaturerule}
\int_\SS p \, d\gs = {1\over n}\sum_{j=1}^n p(v_j), \qquad \forall p\in\Hom(2t),
\end{equation}
where $\gs$ is the normalised surface area measure on the unit sphere $\SS$ in $\Rd$. 
and $\Hom(2t)$ are the homogeneous polynomials $\Rd\to\RR$ of degree $2t$.
The integral of any monomial $x^\ga=x_1^{\ga_1}\cdots x_d^{\ga_d}$ is zero, unless the
power of every coordinate is even, in which case
\begin{equation}
\label{monomialintformula}
\int_\SS x^{2\ga} \, d\gs(x)= {({1\over2})_\ga\over({d\over2})_{|\ga|}}, 
\end{equation}
with $(a)_\ga:=\prod_j a_j(a_j+1)\cdots (a_j+\ga_j-1)$ the Pochammer symbol.

We now consider our design. The cubature rule (\ref{cubaturerule}) for $\Hom(8)$
restricted to the sphere $x^2+y^2+z^2=1$, implies that the monomials $x^2,x^4,x^6,x^8$ are 
integrated, i.e.,
$$  {1\over 8}\sum_j b_j^2={1\over3}, \qquad
{1\over8}\sum_j b_j^4={1\over5}, \qquad
{1\over8}\sum_j b_j^6={1\over7}, \qquad
{1\over8}\sum_j b_j^8={1\over9}, $$
which implies that
$$ \sum_j a_j = {3\over2}\left(\sum_j b_j^2-{8\over3}\right)=0,
\qquad \sum_j a_j^2 
= {9\over4}\left(\sum_j b_j^4-{2\over3}\sum_j b_j^2+{8\over9}\right)
={8\over5}. $$
Since our design has the symmetry group $G=\{I,g,g^2\}$, 
it is sufficient to check the cubature rule holds for the 
polynomials $\Hom(8)^G$, which are invariant under this group,
i.e., the image of $\Hom(8)$ under the Reynolds operator $R_G$ given by
$$ R_G(f) := {1\over|G|}\sum_{g\in G} f^g, \qquad f^g := f(g\cdot). $$
By computing the Molien series 
\begin{align*}
\sum_{g\in G} {1\over\det(I-tg)}
&= \sum_{j=0}^\infty \dim(H(j)^G) t^j \cr
&=1+t+2t^2+4t^3+5t^4+7t^5+10t^6+12t^7+15t^8+19t^9+\cdots,
\end{align*}
we see that $\Hom(8)^G$ has dimension $15$ (we are only concerned with its restriction 
to the sphere, which happens to have the same dimension).
We have
$$\Hom(2)^G=\spam\{x^2,y^2+z^2\}, $$
since $x^2$ (by our choice of $b_j$) and $x^2+y^2+z^2$ (which is $1$ on the sphere) 
are integrated by the cubature rule, so is $\Hom(2)^G$, and hence all of $\Hom(2)$.
We now consider 
$$ \Hom(4)^G=\spam\{x^4,(y^2+z^2)^2,x^2(y^2+z^2), xy(3z^2-y^2),xz(3y^2-z^2)\}. $$
On the sphere $x^2+y^2+z^2=1$, the first three of the polynomials above can be written
as $x^4,(1-x^2)^2,x^2(1-x^2)$ and so are integrated by the cubature rule.
To integrate the fourth polynomial $xy(3z^2-y^2)$, which 
can be written on the sphere as
$$ xy(3z^2-y^2)|_\SS = xy(3-3x^2-4y^2), $$
we must have
$$ {1\over8}\sum_j b_jy_j(3-3b_j^2-4y_j^2) = 0. $$
The fifth polynomial on the sphere cannot be written as a polynomial in $x,y$ only, 
and so we get the condition
$$ xz(3y^2-z^2)|_\SS = xz(3-3x^2-4z^2) \Implies
 {1\over8}\sum_j b_jz_j(3-3b_j^2-4z_j^2) = 0. $$
Continuing in this way, we obtain the following condition.


\begin{theorem} Let
$$ g=\pmat{1&\cr&R}, \quad 
R=\pmat{-{1\over2} &-{\sqrt{3}\over2}\cr {\sqrt{3}\over2} & -{1\over2} }, \qquad
v_j=\pmat{b_j\cr x_j\cr y_j}\in\RR^3, \quad b_j^2+x_j^2+y_j^2=1. $$
Then the orbit of the eight vectors $\{v_1,\ldots,v_8\}$ under the unitary 
action of the group $G=\{I,g,g^2\}$
is a $24$-vector $(4,4)$-design for $\RR^3$ if and only if
$$  {1\over 8}\sum_j b_j^2={1\over3}, \qquad {1\over8}\sum_j b_j^4={1\over5}, \qquad
{1\over8}\sum_j b_j^6={1\over7}, \qquad {1\over8}\sum_j b_j^8={1\over9}, $$
$$ \sum_j b_j^{2k-1}y_j(3-3b_j^2-4y_j^2)=\sum_j b_j^{2k-1}z_j(3-3b_j^2-4z_j^2) = 0, \quad
k=1,2,3, $$
$$ {1\over8} \sum_j b_j^2y_j^2(3-3b_j^2-4y_j^2)^2={8\over 315}, \qquad
 \sum_j b_j^{2k}y_jz_j(3z_j^2-y_j^2)(3y_j^2-z_j^2)=0, \quad k=0,1, $$
$$ \sum_j (y_j^4-z_j^4)(y_j^4-14y_j^2z_j^2+z_j^4) = 0. $$
\end{theorem}

\begin{proof}
A basis for the $\Hom(8)^G$ is given by the $15$ polynomials
$$ x^8, \quad  x^6(y^2+z^2), \quad x^4(y^2+z^2)^2, \quad x^2(y^2+z^2)^3, \quad (y^2+z^2)^4, $$
$$ x^5y(3z^2-y^2), \quad x^3y(3z^2-y^2)(y^2+z^2), \quad xy(3z^2-y^2)(y^2+z^2)^2, $$
$$ x^5z(3y^2-z^2), \quad x^3z(3y^2-z^2)(y^2+z^2), \quad xz(3y^2-z^2)(y^2+z^2)^2, $$
$$ x^2y^2(3z^2-y^2)^2,  \quad x^2yz(3z^2-y^2)(3y^2-z^2), \quad yz(3z^2-y^2)(3y^2-z^2)(y^2+z^2), $$
$$ (y^2-z^2)(y^2+z^2)(y^2-4yz+z^2)(y^2+4yz+z^2) = (y^4-z^4)(y^4-14y^2z^2+z^4). $$
By using $x^2+y^2+z^2=1$ on the sphere to eliminate variables, and taking appropriate
linear combinations to simplify, we obtain the desired equations, e.g.,
the polynomials in the first row restricted to the sphere span the
same subspace as $1,x^2,x^4,x^6,x^8$, which gives the condition
$$ {1\over8}\sum_j b_j^{2k} = \int_\SS x^{2k}\,d\gs(x,y,z) = {1\over 2k+1}, \qquad k=0,1,2,3,4. $$
We omit the case $k=0$, since it automatically holds.
\end{proof}

This gives $19$ equations (the $11$ derived and $b_j^2+y_j^2+z_j^2=1$)
in the $24$ variables $b_j,y_j,z_j$, $1\le j\le 8$. 
We were unable to solve these equations using numerical solvers, however
they are easily seen to hold for the numerical designs we obtained.

\subsection{Spherical $(t,t)$-designs with some structure} 

Here is an example where designs with and without structure
are commonly generated.

\begin{example} 
\label{20point2desR5ex}
The equal-norm $20$-point $(2,2)$-designs in $\RR^5$
seem to split into two types:
\begin{itemize}
\item No apparent structure (repeated angles).
\item Exactly $190$ angles, each repeated $5$ times.
\end{itemize}
Both appear to be continuous families.
Further analysis of the numerical designs with repeated angles
shows each is a union of four sets of five vectors, which have 
just two angles (each occurring five times) and projective symmetry
group the dihedral group $D_5$.
\end{example}

Here is an example of the special situation.

\begin{example}
\label{24point2desR6ex}
The search for equal-norm $24$-point $(2,2)$-designs in $\RR^6$
returns either
\begin{itemize}
\item A set of vectors which is not a design, but does have repeated angles.
This might indicate local minima which are not global minima.
\item A design with repeated angles, specifically what appears to be
${1\over2}$ ($48$ times or $32$ times)
and $0$ ($12$ times or $20$ times).
\end{itemize}
We also note that there are no numerical equal-norm designs with $n=25,26$, 
and the minimiser for $25$ vectors appears to be a unique configuration with repeated
angles.
\end{example}







\begin{example}
\label{23vect(3,3)designR4}
The unweighted $23$-vector $(3,3)$-designs in $\RR^4$ 
seem to have some group structure: $12$ vectors with equal-norms,
which have just three angles between them (appearing with multiplicities
$30$, $30$, $6$).
\end{example}

\begin{example}
\label{41vect(3,3)designR5}
The unweighted $41$-vector $(3,3)$-designs in $\RR^5$
seem to have a unique group structure.
Two sets of $16$ vectors with equal norms (the same in all examples), 
four pairs with equal norms, and one vector with a unique norm 
(the same in all examples).
For those with largest norm, the (normalised) angles are $0$
($48$ times) or ${1\over2}$ ($72$ times)
(a MUB like configuration).
The other $16$ make angles ${1\over5}$ ($80$ times) 
and ${3\over5}$ ($40$ times).  
\end{example}

\begin{example}
\label{union2designs}
A number of spherical $(t,t)$-designs constructed
in \cite{MW20} as unions of two orbits that give lower order designs
appear (from our numerical search) to be optimal. 
These include $(3,3)$-designs of $63$ vectors in $\RR^6$ (orbits
of size $27$ and $36$), $91$ vectors in $\RR^7$ (orbits of size $28$ and $63$),
and a $(4,4)$-design of $16$ vectors in $\RR^3$ (orbits of size $6$ and $10$).
There is also a $(9,9)$-design of $360$ vectors in $\RR^4$
(orbits of size $60$ and $300$). This is always detected in our 
numerical search, which is costly, and so it is assumed to be unique 
and optimal.

\end{example}

\begin{example} 
\label{24vect(4,4)designR3}
The equal-norm $24$-vector $(4,4)$-designs in $\RR^3$ have
$92$ different angles, each appearing three times. They either
involve $3$ or $6$ vectors.
\end{example}

\begin{example}
\label{41vect(7,7)designR3}
There appears to be a unique unweighted $41$-vector $(7,7)$-design in $\RR^3$.
This appears in roughly half the searches.
It consists of $8$ sets of $5$ lines, each with projective symmetry group
the dihedral group of order $10$, together with a single line. 
The sets of $5$ lines present as $2$-angle frames, and can be viewed as nonunitary 
images of the unique harmonic frame of $5$ lines in $\RR^3$ (the lifted five equally
spaced lines in $\RR^2$).
\end{example}

We say that subspaces with orthogonal
projections $P_j$ and $P_k$ are {\bf isoclinic} with {\it angle} $\gs$
if (\ref{isoclinicdefn}) holds.

\begin{example}
\label{isoclinicsubspacesdesign}
The search for equal-norm $(5,5)$-designs for $\RR^4$
(see Figure \ref{special-situation}) provided two examples of the
special situation: a unique putatively optimal design of $60$ points,
and ones with $72$ points. The $72$-point designs
appear to be part of an infinite family.
Each numerical design has projective symmetry group $\ZZ_6$, 
and consists of $12$ orbits of size $6$. These orbits consist of six
equally spaced lines in a plane (two-dimensional subspace).
The $12$ planes in $\RR^4$ appear to have a unique
geometric configuration: each plane is orthogonal to one other, i.e., $\gs=0$, 
and makes the following angles with the other ten
$$  \gs_1^2={5+\sqrt{5}\over10}\approx0.72361\quad\hbox{($5$ times)},
\qquad \gs_2^2={5-\sqrt{5}\over10}\approx0.27640\quad\hbox{($5$ times)}.  $$
\end{example}

\begin{example} 
\label{largetjumps}
For large values of $t$, the jump in the generic case can be less
pronounced, e.g., for $(10,10)$-designs $\RR^3$. A heuristic explanation for this
is that for small angles, the terms $|\inpro{v_j,v_k}|^{2t}$ are close to numerical
zero, e.g., for $|\inpro{v_j,v_k}|\le{1\over3}$ and $t=10$, we have
$$ |\inpro{v_j,v_k}|^{2t}\le (\hbox{${1\over3}$})^{20}\approx 10^{-10}. $$
\end{example}

Motivated by our calculations, we will say that an equal-norm spherical $(t,t)$-design 
of $n$ points for $\RR^d$ is 
{\bf exceptional} if there exists no $(t,t)$-design of $n-1$ or $n+1$ points.
This is an easily checkable condition that can indicate the existence of 
interesting designs.

\begin{example} Of the putatively optimal spherical $(2,2)$-designs for $\Rd$ in Table 1,
those for $d=3,4,7$ are exceptional. There are exceptional $(3,3)$-designs for $d=4,8$.
The $(5,5)$-designs for $\RR^4$ of $60$ and $72$ points 
(Example \ref{isoclinicsubspacesdesign}) are exceptional.
\end{example}

\section{Designs from number theory and cubature}

We now consider some designs first obtained as algebraic formulas,
and a completely new one.

\subsection{The Reznick $11$-point spherical $(3,3)$-design for $\RR^3$}
\label{Rez3.3.11}

The first putatively optimal design on the list of \cite{HW21} for which an 
explicit design was not known is a weighted spherical $(3,3)$-design of $11$ points for $\RR^3$,
which was said to have ``{\it no structure}''. In \cite{BGMPV19} it is referred to 
as the {\it Reznick design}, due to the formula (9.36) of \cite{R92}
\begin{align} 
\label{Reznickformula}
\begin{split}
540 & (x^2+y^2+z^2)^3 
= 378x^6+378y^6+280z^6
+(\sqrt{3}x+2z)^6 +(\sqrt{3}x-2z)^6 \cr
&\qquad+(\sqrt{3}y+2z)^6 +(\sqrt{3}y-2z)^6 +(\sqrt{3}x+\sqrt{3}y+z)^6 
\cr
&\qquad +(\sqrt{3}x-\sqrt{3}y+z)^6 +(\sqrt{3}x+\sqrt{3}y-z)^6 +(\sqrt{3}x-\sqrt{3}y-z)^6. 
\end{split}
\end{align}
Let us elaborate. The definition $f(V)=0$ for being a spherical $(t,t)$-design is equivalent to
the ``Bessel identity'' (see Theorem 6.7 \cite{W18})
\begin{equation}
\label{BesselIdentity}
c_t(\Rd) \norm{x}^{2t}
=  {1\over\sum_{\ell=1}^n \norm{v_\ell}^{2t}} \sum_{j=1}^n(\inpro{x,v_j})^{2t}, \qquad
\forall x\in\Rd, 
\end{equation}
which allows the {\it $d$-ary $2t$-ic form} $\norm{x}^{2t}=(x_1^2+\cdots+x_d^2)^t$ to be
written as a sum of $n$ $2t$-powers. 
The converse is also true, that is if there is a constant $C$ with
\begin{equation}
\label{sumofsquareformula}
C \norm{x}^{2t} =  \sum_{j=1}^n(\inpro{x,v_j})^{2t}, \qquad \forall x\in\Rd,
\end{equation}
then integrating over the unit sphere in $\Rd$ using 
(\ref{monomialintformula}) for $\ga=(t,0,\ldots,0)$ gives
$$ C 
= \sum_{j=1}^n \norm{v_j}^{2t}\int_\SS (\inpro{x,{v_j\over\norm{vj}}})^{2t}\,d\gs(x)
= \sum_{j=1}^n \norm{v_j}^{2t}\int_\SS x_1^{2t}\,d\gs(x)
= c_t(\Rd) \sum_{j=1}^n \norm{v_j}^{2t}, $$
so that $(v_j)$ is a spherical $(t,t)$-design.
Thus
(\ref{Reznickformula}) gives an $11$-point spherical $(3,3)$-design
$$ V=\pmat{ 
\sqrt[6]{378} & 0 & 0 & \sqrt{3} & \sqrt{3} & 0 & 0 & \sqrt{3} & \sqrt{3} & \sqrt{3} & \sqrt{3}  \cr
0 & \sqrt[6]{378} & 0 & 0 & 0 & \sqrt{3} & \sqrt{3} & \sqrt{3} & -\sqrt{3} & \sqrt{3} & -\sqrt{3}  \cr
0 & 0 & \sqrt[6]{280} & 2 & -2 & 2 &-2 & 1 & 1 & -1 & -1 }. $$
Moreover, Theorem 9.28 of \cite{R92} implies that this 
design is optimal. We make some observations/comments based on our calculations:

\begin{itemize}
\item The algebraic variety of (optimal) $11$-point weighted spherical $(3,3)$-designs for $\RR^3$ 
appears to have infinitely many points. 
\item A generic numerical design on it has no symmetry properties, with none of the norms 
$\norm{v_j}$ repeated.
\item The Reznick design has projective symmetry group of order $2$ (exchange the $x$ and $y$ coordinates),
and three different norms taken by $1,2,8$ of the vectors.
\end{itemize}


\subsection{A new $11$-point $(3,3)$-design for $\RR^3$}
\label{new11pointdesign}

The search for numerical $11$-point $(3,3)$-designs for $\RR^3$,
with the condition that two of the vectors have equal norms,
yielded the Reznick design (which appears to be unique) and also, 
frequently, a design with two sets of five vectors with equal norm.
This new design  
has symmetry group the dihedral group $D_5$. 

The projection of each set of five vectors onto the orthogonal complement of
the other single vector gave sets of five equally spaced lines in $\RR^2$,
exactly the same up to a scalar multiple. Thus we came the conjectured 
analytic form of such a design:
\begin{equation}
\label{Vconjform}
V=\pmat{a_1E & a_2E & 0\cr b_1e &-b_2e&-b_3}, \qquad
E=[\pmat{\cos({2\pi\over 5}k)\cr\sin({2\pi\over 5}k)}]_{0\le k\le 4}, \quad
e=[1]_{0\le k\le 4}, 
\end{equation}
where, numerically,
$$ a_1\approx 0.972824, \quad b_1\approx 0.172322, 
\quad a_2\approx 0.736481, \quad b_2\approx 0.692954,
\quad b_3\approx 1.003311, $$
with the normalisation
\begin{equation}
\label{newdesignnormalisation}
5(a_1^2+b_1^2+a_2^2+b_2^2)+b_3^2=11.
\end{equation}
From this assumed structural form, 
by substituting into the sum of squares formula
(\ref{sumofsquareformula}), 
we deduce
the necessary and sufficient conditions for such a design
$$ 5b_1^6+5b_2^6+b_3^6={25\over16}(a_1^6+a_2^6), 
\qquad
a_1^6+a_2^6 = 6(a_1^4b_1^2+a_2^4b_2^2)
= 8(a_1^2b_1^4+a_2^2b_2^4). $$
A fourth more complicated equation is given by the variational condition
$f_{3,3,11}(V)=0$.
In the computer algebra package Maple, we attempted to solve
these four equations for four of the variables $a_1,a_2,b_1,b_2,b_3$, 
with the other as a parameter.
This yields some solutions with complex entries, some
which are real but not numerically correct, and some which are numerically
correct -- often with very complicated formulas. 
With $b_2$ as the free parameter, we eventually came to
$$ {a_1\over b_2} 
= (3051 - 297\sqrt{105})^{1\over6}, \qquad
{a_2\over b_2} 
= { 3\sqrt{5}-\sqrt{21} \over2}
=\left( { 32373- 3159\sqrt{105}\over2}\right)^{1\over6}, $$
$$ {b_1\over b_2} = \left({ 135311 -13205 \sqrt{105}\over 64}\right)^{1\over6},
\qquad
{b_3\over b_2} = \left({ 1246875-121625 \sqrt{105}\over 64}\right)^{1\over6}. $$
On putting these ratios with $b_2$ (presented as sixth roots) into Maple, 
the variational inequality and the sum of squares formula are seen to hold, 
with the sums of the $6$-th powers of the $11$ inner products with $(x,y,z)$ giving
$$  {675\over32} \bigl(1425-139\sqrt{105}\bigr) b_2^6\, (x^2+y^2+z^2)^3. $$
To obtain a neat formula for this design, with $(1425-139\sqrt{105}) b_2^6$ rational, 
we choose
$$ b_2 = \bigl(1425+139\sqrt{105}\bigr)^{1\over6} $$
to get
$$ a_1= \bigl(12960 + 864\sqrt{105}\bigr)^{1\over6}, \quad
a_2= \bigl(12960 - 864\sqrt{105}\bigr)^{1\over6}, $$
$$
b_1= \bigl(1425-139\sqrt{105}\bigr)^{1\over6}, \quad
b_2= \bigl(1425+139\sqrt{105}\bigr)^{1\over6}, \quad
b_3= 26250^{1\over6}. $$
This gives an $11$-point $(3,3)$-design for $\RR^3$ of the form 
(\ref{Vconjform}), which does not
satisfy the normalisation (\ref{newdesignnormalisation}).
The corresponding sum of squares can be written
$$ \sum_{j=0}^1\sum_{k=0}^4 \bigl( \ga_j (c_k x+s_k y) + \gb_j z \bigr)^6
+26250 z^6 = 40500(x^2+y^2+z^2)^3, $$
where $c_k=\cos({2\pi\over 5}k)$, $s_k=\sin({2\pi\over 5}k)$, and
$$ \ga_j=\bigl(12960+(-1)^j 864\sqrt{105}\bigr)^{1\over6}, \qquad
\gb_j=(-1)^j\bigl(1425-(-1)^j 139 \sqrt{105}\bigr)^{1\over6}. $$


This design and the Reznick design appear to be singular points on the
algebraic variety of such designs. It would interesting to study this
variety further, e.g., finding nice points on it (those giving designs with 
large symmetry groups), rational points, or explicitly giving an infinite 
family of these designs.

\subsection{The Reznick/Stroud spherical $(2,2)$-designs for $\RR^4,\RR^5,\RR^6$}
\label{Rez(2,2)-designs}


The equation (9.27)(i) of \cite{R92} is
\begin{align*} 192 (x_1^2+x_2^2+x_3^2+x_4^2)^2 
&= 6(x_1+x_2+x_3+x_4)^4 + \Sigma^4 (3x_1-x_2-x_3-x_4)^4\cr
&\quad +\Sigma^{6}( (1+\sqrt{2})x_1+(1+\sqrt{2})x_2+(1-\sqrt{2})x_3+(1-\sqrt{2})x_4)^4,
\end{align*}
where $\Sigma^4 (3x_1-x_2-x_3-x_4)^4$ stands for the $4$ terms obtained by making a
permutation of the variables $x_1,x_2,x_3,x_4$ in $(3x_1-x_2-x_3-x_4)$, etc.
Therefore
$$ V=\pmat{\sqrt[4]{6}&3&-1&-1&-1&a&a&a&b&b&b\cr
\sqrt[4]{6}&-1&3&-1&-1&a&b&b&a&a&b \cr
\sqrt[4]{6}&-1&-1&3&-1&b&a&b&a&b&a \cr
\sqrt[4]{6}&-1&-1&-1&3&b&b&a&b&a&a}, 
\quad a=1+\sqrt{2}, \ b=1-\sqrt{2},
$$
is an $11$-point spherical $(2,2)$-design for $\RR^4$. 
This is the first in a family of three optimal spherical 
$(2,2)$-designs for $\Rd$, $d=4,5,6$ that can be obtained 
from Stroud's \cite{S71} antipodal cubature rules of degree $5$ ($5$-designs) 
for the unit sphere, given by
$$C(x_1^2+\cdots+x_d^2)^2 = a_1(\Sigma x_j)^4+\Sigma^d(a_2\Sigma x_j+a_3 x_1)^4
+\Sigma^{{d\choose2}}(a_4\Sigma x_j+a_5(x_1+x_2))^4, $$
$$ g=(8-d)^{1\over4}, \quad
a_1=8(g^4-1)(g^2\pm\sqrt{2})^4, \quad
a_2=2g^2\pm2\sqrt{2}, $$
$$ a_3=\mp2\sqrt{2}g^4-8g^2, \quad a_4=2g, \quad a_5=\mp2\sqrt{2}g^3-8g,
\quad C=3a_5^4. $$
The corresponding vectors are
$$ V=[a_1^{1\over4} u,\{a_2u+a_3e_j\}_{1\le j\le d},\{a_4u+a_5(e_j+e_k)\}_{1\le j<k\le d}],
\qquad u=e_1+\cdots+e_d. $$

Our numerical search shows that for $d=4$ there is an infinite family of designs,
and for $d=5,6$ there is a unique design. The unique designs can be 
obtained as a union of two orbits (sizes $6,10$ and $6,16$ respectively) 
\cite{MW20}.



\subsection{Kempner's $24$-point spherical $(3,3)$-design for $\RR^4$}

With the $\pm$ independent of each other and the previous notation, Kempner 1912 gives
\begin{equation}
\label{Kempnerformula}
120 (x_1^2+x_2^2+x_3^2+x_4^2)^3
= \Sigma^4 (2x_1)^6+8\Sigma^{12}(x_1\pm x_2)^6
+\Sigma^8 (x_1\pm x_2\pm x_3\pm x_4)^6. 
\end{equation}
The vectors in the corresponding design have equal norms.

\begin{example}
\label{Kempnerexample}
Our search for equal-norm 
$24$-point spherical $(3,3)$-designs for $\RR^4$ gave an
infinite family, with repeated angles $0$ ($60$ times),
 ${1\over\sqrt{2}}$ ($32$ times) and $21$ other angles (each $4$ times). 
The exact example given by
(\ref{Kempnerformula}) has just three angles 
$0$ ($108$ times) ${1\over\sqrt{2}}$ ($72$ times)
and ${1\over2}$ ($96$ times), and is the orbit of three vectors
under the natural action of $S_4$. 
This design can also be obtained
as the orbit of two vectors under the action of the real reflection group
$W(F_4)$ (Shephard-Todd number $28$) \cite{MW20}.

Interestingly, our search for equal-norm
$23$-point spherical $(3,3)$-designs for $\RR^4$ seemed to give a unique 
configuration, with repeated angles.
\end{example}


\subsection{Kotelina and Pevnyi's $120$-point 
$(3,3)$-design for $\RR^8$}
 
A similar formula to (\ref{Kempnerformula}) is given in
\cite{KP11}, which leads to what appears to be the unique optimal 
$120$-vector $(3,3)$-design for $\RR^8$.

\begin{example}
\label{Kotelinaexample}
Our numerical search for equal-norm spherical $(3,3)$-designs for $\RR^8$ 
gave an optimal design of $120$ points, with repeated angles
$0$ ($3780$ times) and ${1\over2}$ ($3360$ times). This is easily
recognised to be the design of \cite{KP11}. This is an example of the
special situation, see Figure \ref{special-situation_3_8}. The next 
value of $n$ for which there is a design is $n=250$.
\end{example}

Intuitively, one would expect that for fixed $t$, the numbers
$n_e=n_e(d)$ and $n_w=n_w(d)$ of vectors in optimal 
equal-norm and weighted $(t,t)$-designs in $\Rd$
should be increasing functions of $d$. The above example for 
$t=3$ is so exceptional that it provides a counter example
(see Table \ref{real(t,t)-designs}), i.e.,
$$ n_e(7)= 158, \qquad n_e(8)= 120, \qquad n_e(9)= 380. $$
Nevertheless, we expect $d\mapsto n_e(d)$ and 
$d\mapsto n_w(d)$ to be asymptotically increasing.


\begin{figure}[ht]
\centering
\includegraphics[width=7.5cm]{\figures 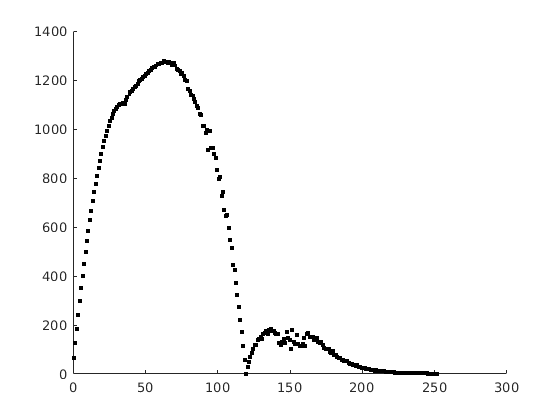}
\includegraphics[width=7.5cm]{\figures 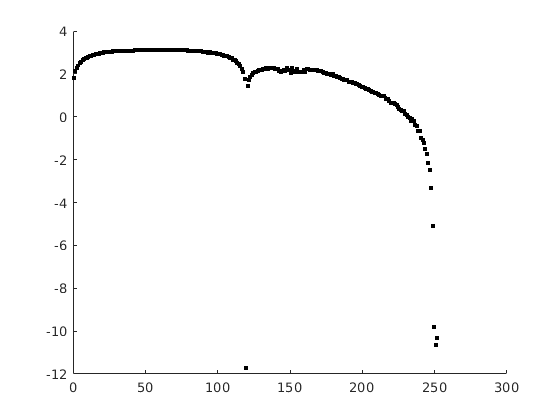}
\caption{Special situation: 
the graphs $n\mapsto f_{t,d,n}$ and $n\mapsto \log_{10} f_{t,d,,n}$
for $t=3$, $d=8$.}
\label{special-situation_3_8}       
\end{figure}

\section{Conclusion}

We used the {\tt manopt} optimisation software to considerably
enlarge the list of putatively optimal spherical $(t,t)$-designs
(see Table \ref{real(t,t)-designs}). 

The generic situation seems to be 

\begin{itemize}
\item The algebraic variety of optimal equal-norm and weighted spherical 
$(t,t)$-designs for $\Rd$ has positive dimension. A typical 
(numerical) element in the variety has little or no structure/symmetry, 
though there may be such points on the variety. It is characterised 
by a significant ``jump'' down in the error $f_{t,d,n}$ to numerical
zero from configurations with fewer points.
\end{itemize}
A highlight is the first full geometric description of the algebraic variety
of optimal designs in the generic case of positive dimension 
(Theorem \ref{isoclinicMercedes}).
The new $11$-point $(3,3)$-design for $\RR^3$ of $\S\ref{new11pointdesign}$ 
is also of particular interest.

There are also special situations 
(previously the only explicit examples known)
where 


\begin{itemize}
\item The algebraic variety consists of single point, or a finite set of points.
These designs have a high degree of structure/symmetry, which may lead
to explicit constructions (as group orbits).
\end{itemize}

Some other observations about our methods that may be of future use are:

\begin{itemize}
\item The special situation is not always detected by a single calculation
(though it might be indicated by the graph of the minimum values obtained),
and so could be missed by methods which do only one calculation for a 
given value of $n$. 
\item In the generic situation, the jump value of $n$ is not
always detected by one calculation, e.g., for
unweighted $(t,d,n)= (4,5,101),(6,4,116)$
and weighted $(2,9,45),(6,4,154),(8,3,78)$. 
Because of this, we suggest constructing
several numerical designs of $n-1$ points, where $n$ is the presumed jump
(optimal value).
\item Adaptive methods could be used to find larger numerical designs,
e.g., instead of calculating $f_{t,d,n}$ for consecutive values of $n$
until numerical zero is found, a bisection method could be used to find
the ``jump'', or an interval 
in which it lies. 
\item The methods outlined apply to a wide class of configurations,
and could for example be applied to the Game of Sloanes, 
for which the optimal solutions are strictly speaking not an algebraic
variety.
\item We came to no firm conclusions about the existence of local minimisers 
which are not optimal designs (for $t>1$), when an optimal design exists.
\end{itemize}

\section{Acknowledgements}

Thanks to Josiah Park for several very useful discussions, 
and for alerting us to the work of Rezick \cite{R92}.
Thanks to Dustin Mixon for suggesting the manopt software.



\bibliographystyle{alpha}
\bibliography{references}
\nocite{*}
 
\vfil
\end{document}